\documentclass[a4paper,11pt]{article}
\usepackage{amsmath, amsfonts, amscd, amssymb, amsthm, enumerate, xypic}

\DeclareMathOperator{\id}{\operatorname{id}}

\DeclareMathOperator{\End}{\operatorname{End}}

\DeclareMathOperator{\Ker}{\operatorname{Ker}}

\renewcommand{\setminus}{\smallsetminus}
\renewcommand{\epsilon}{\varepsilon}

\def\R{\mathbb{R}}
\def\C{\mathbb{C}}

\def\N{\mathbb{N}}

\def\calD{\mathcal{D}}

\def\calQ{\mathcal{Q}}

\def\calV{\mathcal{V}}

\theoremstyle{definition}

\theoremstyle{plain}
\newtheorem{theo}{Theorem}
\newtheorem{prop}[theo]{Proposition}

\newtheorem{lemma}[theo]{Lemma}

\theoremstyle{plain}

\theoremstyle{remark}
\newtheorem{Rems}{Remarks}[section]
\newtheorem{Rem}[Rems]{Remark}

\title{A Convex-Analytical Proof of the Fundamental Theorem of Algebra}
\author{Cl\'ement de Seguins Pazzis\footnote{Universit\'e de Versailles Saint-Quentin-en-Yvelines, Laboratoire de Math\'ematiques
de Versailles, 45 avenue des \'Etats-Unis, 78035 Versailles cedex, France}
\footnote{e-mail address: clement.de-seguins-pazzis@ac-versailles.fr}}

\begin{document}

\thispagestyle{plain}


\maketitle
\begin{abstract}
A weak version of Birkhoff's generalization of the Perron-Frobenius theorem
states that every endomorphism of a finite-dimensional real vector that leaves invariant a non-degenerate closed convex cone
has an eigenvector in that cone.

Here, we show that this theorem, whose proof relies only upon basic convex analysis, yields very short proofs of both the spectral theorem for selfadjoint operators
of Euclidean spaces and the Fundamental Theorem of Algebra.
\end{abstract}



Let $E$ be a finite-dimensional real vector space.
We define a \textbf{cone} $C$ as a closed subset of $E$ that contains the zero vector and is stable under addition of vectors and under multiplication by positive scalars.
In particular, every cone is convex. The cone $C$ is called \textbf{nondegenerate} if $C \cap (-C)=\{0\}$.
Finally, we say that $C$ is \textbf{nonzero} if $C$ contains a nonzero vector.

A classical result of Birkhoff generalizes weak forms of the Perron-Frobenius theorem as follows:

\begin{theo}[Weak form of Birkhoff's theorem \cite{Birkhoff}]
Let $C$ be a nonzero nondegenerate cone of a finite-dimensional real vector space $E$.
Let $u$ be an endomorphism of $E$ such that $u(C) \subseteq C$.
Then $u$ has an eigenvector in $C$, with nonnegative associated eigenvalue.
\end{theo}

Here, the nonnegativity of the eigenvalue is an obvious consequence of the nondegeneracy of the cone, and the difficult statement is the existence of an eigenvector in $C$ for $u$. Birkhoff's theorem is actually a special case of a much more difficult theorem of Krein and Rutman \cite{KreinRutman}, which deals with cones in Banach spaces and
predates Birkhoff's publication by several years.

Birkhoff's theorem requires that $C$ has nonempty interior in $E$, but it is easy to derive the above formulation from his theorem,
as $C$ always has nonempty interior in its linear span in $E$. The strong form of Birkhoff's theorem asserts that, if in addition $C$ has nonempty interior in $E$ then there exists an eigenvector of $u$ in $C$ whose associated eigenvalue is the spectral radius of $u$: this strong form is however of no interest to us here.
Birkhoff's original proof \cite{Birkhoff} used the complex Jordan canonical form of the operator $u$,
but as we shall see in the appendix of this note it is possible to give a purely convex-analytical proof of the weak version that requires
no knowledge of canonical forms for endomorphisms of a vector space, and even avoids introducing complex numbers at all, just as many proofs of the Perron-Frobenius theorem avoid such things and rely only upon basic topology of finite-dimensional real vector spaces.

Here, we shall apply Birkhoff's theorem in a somewhat unusual way to
derive both the spectral theorem for selfadjoint operators in Euclidean spaces,
and the Fundamental Theorem of Algebra. This method is remarkable in three ways:
it uses a strategy that largely differs from the known ones, it completely avoids
introducing complex numbers at all, and finally it uses Birkhoff's theorem with a smooth cone rather than with the polyhedral cones
the Perron-Frobenius theorem is traditionally associated with.

Let us now state the results we shall derive from Birkhoff's theorem:

\begin{theo}[Weak form of the spectral theorem]
Let $q$ be a definite positive quadratic form on a real vector space $E$ with positive finite dimension $d>0$,
and $u$ be a $q$-selfadjoint endomorphism of $E$. Then $u$ has a (real) eigenvalue.
\end{theo}

\begin{theo}[Real form of the Fundamental Theorem of Algebra]
Every irreducible polynomial over the real numbers has degree $1$ or $2$.
\end{theo}

That the real form of the Fundamental Theorem of Algebra implies the complex one is easy, since if we take a nonconstant polynomial $p(t) \in \C[t]$,
then $p(t)\overline{p(t)}$ has real coefficients, and it only remains to observe that every polynomial with degree $2$ and real coefficients has a complex root.

Our proofs of both the above results are based on the following idea: let $E$ be a nonzero finite-dimensional real vector space.
The space $\calQ(E)$ of all quadratic forms on $E$ is finite-dimensional, and we denote by
$\calQ^+(E)$ the subset of all semi-definite positive quadratic forms:
it is very elementary to check that it is a nonzero and nondegenerate cone,
with closeness proved simply by remarking that the evaluation mappings $q \in \calQ(E) \mapsto q(x)$, with $x \in E$, are all continuous
because they are linear and $\calQ(E)$ is finite-dimensional.

The key in both proofs is to scrutinize the action of linear operators on quadratic forms:
to every endomorphism $u \in \End(E)$, we associate the linear mapping
$$Q(u) : q \in \calQ(E) \mapsto q \circ u \in \calQ(E),$$
and a critical yet mundane observation is that $Q(u)$ leaves the cone $\calQ^+(E)$ invariant.

We are now ready for the proofs of the two stated theorems. The proof of the second one will require the validity of the first one.

\begin{proof}[Proof of the weak form of the spectral theorem]
Consider the subspace $\calQ_u(E)$ of all quadratic forms $q'$ on $E$ for which
$u$ is $q'$-selfadjoint. We note that $\calQ_u(E)$ is invariant under $Q(u)$:
indeed, for any $q' \in \calQ_u(E)$, with corresponding symmetric bilinear form $B$, the
symmetric bilinear form which corresponds to $q \circ u$ is $B' : (x,y) \mapsto B(u(x),u(y))$,
and clearly
$$\forall (x,y)\in E^2, \quad B'(u(x),y)=B(u^2(x),u(y))=B(u(x),u^2(y))=B'(x,u(y)).$$
Hence the invariance under $Q(u)$ of the intersection $\calQ^+_u(E):=\calQ^+(E) \cap \calQ_u(E)$, which is obviously a nondegenerate cone in $\calQ_u(E)$, being
the intersection of a nondegenerate cone with a linear subspace; Moreover $\calQ_u^+(E)$ is nonzero because we have assumed that $u$ is selfadjoint for some
definite positive quadratic form on $E$.

By Birkhoff's theorem there is an eigenvector $q'$ for $q'' \mapsto q'' \circ u$ in the cone $\calQ^+_u(E)$, with nonnegative associated eigenvalue:
hence $q'$ is a semi-definite positive quadratic form on $E$ and there is a real number $\lambda \geq 0$ such that $q' \circ u=\lambda\,q'$.
In denoting by $B'$ the polar form of $q'$, we deduce that
$$\forall (x,y)\in E^2, \; B'(u^2(x),y)=B'(u(x),u(y))=\lambda B'(x,y),$$
and hence $u^2-\lambda\,\id_E$ has its range included in the radical of $B'$ (i.e.\ the orthogonal of $E$ under $B'$), which differs from $E$.
Hence at least one of $u-\sqrt{\lambda} \id_E$ and $u+\sqrt{\lambda} \id_E$ is not an isomorphism, to the effect that 
at least one of $\sqrt{\lambda}$ and $-\sqrt{\lambda}$ is an eigenvalue of $u$.
\end{proof}

Below is a reformulation of the proof in terms of selfadjoint endomorphisms.

\begin{proof}[Reformulated proof of the weak form of the spectral theorem]
Let us consider $E$ with its structure of Euclidean space inherited from $q$, and denote by $\langle -,-\rangle$
the associated inner product.
This defines a notion of selfadjoint endomorphisms of $E$, and of semi-definite positive selfadjoint endomorphisms.
Let us consider the vector space $\calV$ of all endomorphisms $f$ of $E$ that are selfadjoint and commute with $u$.
In $\calV$, we have the cone $\calV^+$ of all semi-definite positive endomorphisms. Clearly $\calV^+$ is a non-degenerate convex cone,
and it is nonzero since it contains $\id_E$.

For all $f \in \calV$ we note that $u^2 f=u^\star f u$ is in $\calV$, and if in addition $f \in \calV^+$
then $u^2 f \in \calV^+$ (indeed, for all $x \in E$ we have $\langle x,(u^2 f)(x)\rangle=\langle u(x),f(u(x))\rangle \geq 0$ thanks to the commutation of $u$ with $f$ and to the selfadjointness of $u$).

Thus, we can apply Birkhoff's theorem to the linear mapping $f \in \calV \mapsto u^2 f \in \calV$.
This yields a nonzero element $f \in \calV^+$ and a nonnegative real number $\lambda$ such that $u^2f=\lambda f$.
Thus $u^2-\lambda \id_V$ is non-invertible, and the conclusion is obtained just like in the previous proof.
\end{proof}

\begin{proof}[Proof of the real form of the Fundamental Theorem of Algebra]
Let $p(t) \in \R[t]$ be a real irreducible polynomial with degree $d \geq 1$.
We can find a finite-dimensional real vector space $E$ equipped with an irreducible\footnote{I.e., with no nontrivial invariant subspace.} endomorphism $u$ with minimal polynomial $p(t)$: for example, we take the quotient space $E:=\R[t]/(p(t))$ and the endomorphism $v$ of left multiplication by the coset of $t$;
it is easily seen that $p(t)$ is the minimal polynomial of $v$, and then we can restrict $v$ to an arbitrary minimal nonzero invariant subspace\footnote{In fact it can be proved that $v$ is already irreducible, but this precision is useless here.}.

Birkhoff's theorem applied to the endomorphism $Q(u)$ and the cone $\calQ^+(E)$
yields a nonzero semi-definite positive quadratic form $q$ on $E$ and a real number $\lambda \geq 0$ such that $q \circ u=\lambda\,q$.
The isotropy cone $q^{-1}\{0\}=\{x \in E : q(x)=0\}$ is then invariant under $u$, yet because $q$ is semi-definite positive this cone is actually a linear subspace of $E$: this is a classical consequence of the Cauchy-Schwarz inequality for semi-definite positive quadratic forms, which is essentially the only basic fact of quadratic forms theory that we use in our proof.

By the irreducibility of $u$, we deduce that $q$ is definite positive and that $\lambda>0$.
Replacing $u$ with $\lambda^{-1/2} u$, we can then assume that $q \circ u=q$, i.e.\ $u$ is an isometry for $q$. Now, the trick is to consider the $q$-selfadjoint
endomorphism
$$u_+:=u+u^\star=u+u^{-1}\cdot$$
By the weak form of the spectral theorem, $u_+$ has a real eigenvalue $\alpha$, and we consider the corresponding eigenspace $V$.
Since $u$ obviously commutes with $u_+$, we deduce that $V$ is $u$-invariant, and we conclude that $V=E$ since $u$ is irreducible.
Hence $u_+=\alpha\id_E$, to the effect that $t^2+1-\alpha t$ annihilates $u$. Therefore $d\leq 2$.
\end{proof}

\appendix

\section{Appendix: A convex-analytical proof of the weak Birkhoff theorem}

Since Birkhoff's theorem is at heart a theorem of convex analysis and since we used it to prove
the Fundamental Theorem of Algebra, we feel compelled to give a proof of Birkhoff's theorem that is entirely rooted
in convex analysis and avoids using the complex numbers at all. As we shall see, the proof is quite natural.

\subsection{Basics on cones}

Remember that $E$ denotes a finite-dimensional real vector space.
Throughout, we fix a norm $\|-\|$ on $E$ and denote by $S$ the associated unit sphere. We denote by $\R_+$ the set of all nonnegative real numbers.

Our prerequisites are the basic results on the topology of finite-dimensional real vector spaces:
on such a space, all the norms are equivalent, all the linear maps are continuous, all the bilinear mappings
from a product of finite-dimensional vector spaces are continuous, and the compact subsets are the closed bounded ones.
Finally, we require basic theorems on compact (metric) spaces, such as the descending chain theorem:
the intersection of any non-increasing sequence of nonempty compact spaces is nonempty.

Apart from that, the main key is the Hahn-Banach geometric theorem for closed convex subsets (see e.g.\ theorem 1.7 in \cite{Brezis},
where we take the convex compact subset $B$ limited to a single point):

\begin{theo}[Hahn-Banach separation theorem for closed convex subsets]
Let $X$ be a nonempty closed convex subset of $E$, and let $a \in E \setminus X$. Then there exists a linear form
$\varphi$ on $E$ such that $\forall x \in X, \; \varphi(x) > \varphi(a)$.
\end{theo}

From there, we move to the duality of cones. Let $C$ be a cone of $E$, and denote by $E'$ the dual vector space of $E$.
We define the dual cone $C^\star$ as the subset of all linear forms $\varphi \in E'$ such that
$\forall x \in C, \; \varphi(x) \geq 0$. Because the evaluation mappings $\varphi \mapsto \varphi(x)$ are linear, they are continuous,
and it is then easy to see that $C^\star$ is a cone. Conversely, given a cone $\calD$ of $E'$, we define the predual cone
${}^\star \calD$ as the set of all $x \in E$ such that $\forall \varphi \in \calD, \; \varphi(x) \geq 0$.
Again, it is a cone of $E$. Then we have two key properties:

\begin{prop}
One has ${}^\star(C^\star)=C$ for every cone $C$ of $E$.
\end{prop}

\begin{proof}
The inclusion $C \subseteq {}^\star(C^\star)$ is obvious.

Next, let $a \in E \setminus C$. Since $C$ is a closed convex subset, the Hahn-Banach theorem yields a linear form $\varphi \in E'$
such that $\varphi(x) >\varphi(a)$ for all $x \in C$.
Since $0 \in C$ we deduce that $0>\varphi(a)$. Let $x \in C$. Then $t \varphi(x)=\varphi(tx) \geq \varphi(a)$ for all $t \in \R_+$, which requires that $\varphi(x) \geq 0$.
As $\varphi \in C^\star$ and $\varphi(a)<0$, we deduce that $a\not\in {}^\star(C^\star)$.
\end{proof}

\begin{lemma}\label{interiorlemma}
Let $C$ be a cone of $E$. The set $\{\varphi \in E' : \forall x \in C \setminus \{0\}, \; \varphi(x)>0\}$ is open in $E'$.
\end{lemma}

\begin{proof}
Denote by $X$ the complementary subset in $E'$ of the said set. Let us prove that $X$ is closed in $E'$.

Consider a sequence $(f_n)_{n \geq 0}$ in $X$, converging to some $f \in E'$. For every integer $n \geq 0$ there exists a
vector $x_n \in C \setminus \{0\}$ such that $f_n(x_n) \leq 0$. Since $C$ is a cone we can normalize each $x_n$
and reduce the situation to the one where $x_n$ belongs to the unit sphere $S$. Then, since $S$ is compact we can extract a subsequence $(x_{\varphi(n)})$
that converges to some $x_\infty \in S$. Because $C$ is closed we have $x_\infty \in C$, and since the bilinear mapping
$(g,x) \mapsto g(x)$ is continuous we obtain $f(x_\infty) \leq 0$, and hence $f \in X$.
\end{proof}

The last item we need from the theory of cones is the existence of an extremal vector.
Let $C$ be a cone. A vector $x \in C$ is called \textbf{extremal in $C$} when it is nonzero and, for all
$x,y,z$ in $C$, the equality $x=y+z$ implies $y \in \R_+ x$.

\begin{lemma}\label{extremalpointlemma}
Let $C$ be a closed nonzero cone of $E$.
\begin{enumerate}[(i)]
\item If $\dim E>1$ and $C \neq E$ then the boundary $\partial C$ contains a nonzero vector.
\item If $C$ is nondegenerate then it has an extremal vector.
\end{enumerate}
\end{lemma}

\begin{proof}
The proof of the first point is easy : assume that $\dim E>1$ and $C \neq E$. Choose $x \in C \setminus \{0\}$.
There must exist $y \in E \setminus C$ outside of $\R x$,
otherwise $E \setminus \R x \subseteq C$, and since $C$ is closed this would lead to $C=E$.
Now, choosing $y$ we observe that $I:=\{t \in \R : x+ty \in C\}$ is a closed nonempty interval in $\R$ because $C$ is closed and convex, and it must differ from $\R$
because otherwise $y=\lim_{t \rightarrow +\infty} \frac{1}{t} (x+ty)$ would belong to $C$.
Hence we can pick a real bound $s$ of $I$, and it is clear then that $x+sy \in \partial C$ and
$x+sy \neq 0$.

We prove the second point by induction. The result is obvious if $\dim E=1$, for any nonzero vector of $C$ is then extremal.
Assume now that $\dim E>1$ and that $C$ is nondegenerate (in particular $C \neq E$). By double-duality, this requires that $C^\star \neq \{0\}$ and $C^\star \neq E$.
By the first point, the dual cone $C^\star$ has a nonzero boundary point $\varphi$. Set $H:=\Ker \varphi$.
By Lemma \ref{interiorlemma}, there exists a vector $x \in C \setminus \{0\}$ such that $\varphi(x) \leq 0$, and hence $\varphi(x)=0$. Hence the nondegenerate cone
$C \cap H$ of $H$ is also nonzero, and then by induction there is an extremal vector $x$ in the cone $C \cap H$.
Finally, we note that $x$ is extremal in $C$: let indeed $y,z$ in $C$ be such that $x=y+z$. Then $\varphi(y)+\varphi(z)=\varphi(x)=0$
with $\varphi(y) \geq 0$ and $\varphi(z) \geq 0$, and hence $\varphi(y)=\varphi(z)=0$, leading to $y \in H$ and $z \in H$. We conclude that $y \in \R_+ x$ because $x$ is extremal in $C \cap H$.
\end{proof}

Our last key result on cones deals with intersections. First of all, note that
the intersection of an arbitrary (nonempty) family of cones of $E$ is a cone of $E$.

\begin{lemma}[Descending chain lemma for cones]\label{descendinglemma}
Let $(C_n)_{n \geq 0}$ be a non-increasing sequence of nonzero cones of $E$.
Then $\underset{n \geq 0}{\bigcap} C_n$ is nonzero.
\end{lemma}

\begin{proof}
For each $n \geq 0$, the intersection $C_n \cap S$ is a closed subset of $S$, and hence it is compact, besides it is nonempty.
The sequence $(C_n \cap S)_{n \geq 0}$ is non-increasing, and hence its intersection is nonempty. Finally, any point of this intersection is a nonzero element of
$\underset{n \geq 0}{\bigcap} C_n$.
\end{proof}

\subsection{Proof of the weak Birkhoff theorem}

We are almost ready to prove the weak Birkhoff theorem.

\begin{lemma}\label{lemma:nonzerocone}
Let $C$ be a nonzero cone and $u$ be an endomorphism of $E$ such that $\forall x \in C \setminus \{0\}, \; u(x) \neq 0$.
Then $u(C)$ is a nonzero cone.
\end{lemma}

\begin{proof}
It is clear that $u(C)$ contains a nonzero vector and, since $u$ is linear, that it is invariant under addition and multiplication with positive scalars.
The only nontrivial property we must check is that $u(C)$ is closed in $E$.
To see this, note that $C=\R_+ (S \cap C)$ by normalizing, and hence $u(C)=\R_+ u(S \cap C)$ by the linearity of $u$.
The subset $K:=u(S \cap C)$ is compact (since $S \cap C$ is compact, being closed in $S$, and $u$ is continuous) and does not contain $0_E$ because of our assumption on $u$. Now, take sequences $(s_n)_{n \geq 0} \in (\R_+)^\N$ and $(x_n)_{n \geq 0} \in K^\N$ such that
$(s_n x_n)_n$ converges to some $y$ in $E$, and let us prove that $y \in \R_+ K$. By extracting a subsequence, we lose no generality in assuming that
$(x_n)_n$ converges to some $x_\infty \in K$. Then by taking norms we find that $(s_n \|x_n\|)_n$ converges to $\|y\|$, and hence
$(s_n)_n$ converges to $\frac{\|y\|}{\|x_\infty\|}$, which in turn shows that $y=\frac{\|y\|}{\|x_\infty\|} x_\infty \in \R_+ K$.
\end{proof}

Our last key is the following lemma:

\begin{lemma}\label{infiniteintersectionlemma}
Let $C$ be a nonzero cone and $u$ be an endomorphism of $E$ such that $u(C \setminus \{0\}) \subseteq C \setminus \{0\}$.
Then $u^\infty(C)=\underset{n \geq 0}{\bigcap} u^n(C)$ is a nonzero cone and $u(u^\infty(C)))=u^\infty(C)$.
\end{lemma}

\begin{proof}
By induction, the previous lemma yields that $u^n(C)$ is a nonzero cone for every integer $n \geq 0$, and the sequence
$(u^n(C))_{n \geq 0}$ is obviously nonincreasing. Hence $u^\infty(C)$ is a nonzero cone by Lemma \ref{descendinglemma}, while clearly
$u(u^\infty(C))) \subseteq u^\infty(C)$.

Now, take a normalized vector $x \in u^\infty(C) \cap S$. For every integer $n \geq 0$ we write $x \in u(u^n(C))$ and then find some $x_n \in u^n(C) \cap S$ and some $s_n \in \R_+$ such that $s_n u(x_n)=x$. We can extract a subsequence $(x_{\varphi(n)})$ that converges to some $y \in S$.
Because $(u^n(C))_{n \geq 0}$ is non-increasing, the vector $y$ belongs to the closure $\overline{u^n(C)}$, and hence to $u^n(C)$, for every integer $n \geq 0$.
Hence $y \in u^\infty(C)$.
Finally, by taking norms we see that $s_{\varphi(n)}\, \|u(x_{\varphi(n)})\| \rightarrow \|x\|=1$, and just like in the proof of Lemma \ref{lemma:nonzerocone}
we deduce that $x=\|x\|\, u(y)=u(\|x\|\, y)$ with $\|x\|\,y \in u^\infty(C)$. Hence $u(u^\infty(C)))=u^\infty(C)$.
\end{proof}

We can now complete the proof of the weak Birkhoff theorem.
We take a nonzero nondegenerate cone $C$ of $E$ and an endomorphism $u$ of $E$ such that $u(C) \subseteq C$.
The trick is to consider the operator $v:=\id_E +u$, which clearly satisfies $v(C) \subseteq C$
and also $\forall x \in C \setminus \{0\}, \; v(x) \neq 0$ because $C$ is nondegenerate.
We apply Lemma \ref{infiniteintersectionlemma} to find that $C':=v^\infty(C)$ is a nonzero subcone of $C$ that is globally invariant under $v$,
i.e.\ $v(C')=C'$.
In particular, $C'$ is nondegenerate. Next, we observe that $C'$ is invariant under $u$, which easily follows from the commutation of $u$ with $v=u+\id$
and from the fact that $u(C) \subseteq C$.

And now for the \emph{coup de gr\^ace}: by Lemma \ref{extremalpointlemma} the cone $C'$ has an extremal vector $x$.
Then $x=v(y)=y+u(y)$ for some $y \in C'$ (obviously nonzero) by the global invariance, whereas $u(y) \in C'$ because of the invariance of $C'$ under $u$.
Then $u(y) \in \R_+ x$ and hence $u(y)$ and $y$ are linearly dependent. Therefore $y$ is an eigenvector of $u$ that belongs to $C$. QED

\begin{Rem}
By allowing the Axiom of Choice, we can simplify the proof as follows. First of all, let
$(C_i)_{i \in I}$ be a family of nonzero $u$-invariant subcones of $C$,
and assume that it is totally ordered for the inclusion of subsets.
Proceeding as in the proof of Lemma \ref{descendinglemma}, we can use the Borel-Lebesgue property for the compact space $S$
to see that the cone $\underset{i \in I}{\bigcap} C_i$ is nonzero.
Note that this cone is obviously $u$-invariant.

By Zorn's lemma, we deduce that among the $u$-invariant nonzero subcones of $C$, there is a minimal one $C'$.
Then as in the above proof we take $v:=u+\id_E$ and we use Lemma \ref{lemma:nonzerocone} to see that $v(C')$ is a nonzero subcone of $C'$ because $C$ is nondegenerate.
It is easy to observe that $v(C')$ is also $u$-invariant. By the minimality of $C'$, we recover $v(C')=C'$. Then we take an extremal vector $x$ of $C'$ and work as before to see that $x$ is an eigenvector of $u$.
\end{Rem}

\end{document}